\documentclass[a4paper,12pt]{article}
    \usepackage[top=2.5cm,bottom=2.5cm,left=2.5cm,right=2.5cm]{geometry}
    \usepackage{amsmath,amsthm,amsfonts,amssymb,amscd,cite,placeins,mathrsfs}
    \usepackage[margin=1cm,%
                font=small,%
                format=hang,%
                labelsep=period,%
                labelfont=bf]{caption}
\usepackage{graphicx}
\allowdisplaybreaks
\usepackage{longtable,verbatimbox,seqsplit}
\allowdisplaybreaks

\newtheorem{proposition}[subsection]{Proposition}
\newtheorem{theorem}[subsection]{Theorem}

\begin{document}

\baselineskip=0.30in

\vspace*{25mm}
 \begin{center}

 {\LARGE \bf On the Complementary Equienergetic Graphs}\\

  \vspace{5mm}

 {\large \bf Akbar Ali}$^{1,2,}\footnote{Corresponding
 author}$, {\large \bf  Suresh Elumalai}$^{3}$, {\large \bf Toufik Mansour}$^{3}$,\\ {\large \bf Mohammad Ali Rostami}$^{4}$

\vspace{3mm}

\baselineskip=0.20in

 $^1${\it Knowledge Unit of Science,\\
  University of Management and Technology, Sialkot, Pakistan}\\
$^2${\it College of Sciences,\\
  University of Hail, Hail, Saudi Arabia}\\
 {\rm e-mail:} {\tt akbarali.maths@gmail.com}\\[2mm]
 $^3${\it Department of Mathematics, \\
 University of Haifa, 3498838 Haifa, Israel\/} \\
 {\rm e-mail:} {\tt sureshkako@gmail.com, tmansour@univ.haifa.ac.il}\\[2mm]
$^4${\it Institute for Computer Science,\\
  Friedrich Schiller University Jena, Germany\/} \\
 {\rm e-mail:} {\tt a.rostami@uni-jena.de}

 \vspace{6mm}

 (Received July 30, 2019)
 \end{center}

\begin{abstract}
\noindent
Energy of a simple graph $G$, denoted by $\mathcal{E}(G)$, is the sum of the absolute values of the eigenvalues of $G$. Two graphs with the same order and energy are called equienergetic graphs. A graph $G$ with the property $G\cong \overline{G}$ is called self-complementary graph, where $\overline{G}$ denotes the complement of $G$. Two non-self-complementary equienergetic graphs $G_1$ and $G_2$ satisfying the property $G_1\cong \overline{G_2}$ are called complementary equienergetic graphs.
Recently, Ramane \textit{et al.} [Graphs equienergetic with their complements, MATCH Commun. Math. Comput. Chem. 82 (2019) 471--480]
initiated the study of the complementary equienergetic regular graphs
%
and they asked to study the complementary equienergetic non-regular graphs.
In this paper, by developing some computer codes and by making use of some software like Nauty, Maple and GraphTea, all the complementary equienergetic graphs with at most 10 vertices as well as all the members of the graph class $\Omega=\{G \ : \ \mathcal{E}(L(G)) = \mathcal{E}(\overline{L(G)}) \text{, the order of $G$ is at most 10}\}$ are determined, where $L(G)$ denotes the line graph of $G$. In the cases where we could not find the closed forms of the eigenvalues and energies of the obtained graphs, we verify the graph energies using a high precision computing (2000 decimal places) of Maple. A result about a pair of complementary equienergetic graphs is also given at the end of this paper.\\[2mm]

\end{abstract}

\baselineskip=0.30in

\section{Introduction}

Throughout this paper, by a graph we mean a simple graph (that is, a graph without loops and multiple edges).
Vertex set and edge set of a graph $G$ are denoted by $V(G)$ and $E(G)$, respectively. Edge connecting the vertices $u,v\in V(G)$ is denoted by $uv$. An $n$-vertex graph is a graph with order $n$. The complement $\overline{G}$ of a graph $G$ is the graph with the vertex set $V (\overline{G}) = V (G)$
and $uv\in E(\overline{G})$ if and only if $uv\not\in E(G)$.  A graph $G$ with the property $G\cong \overline{G}$ is called self-complementary graph.
The line graph $L(G)$ of a graph $G$ is a graph whose vertex set is $E(G)$ and two vertices of $L(G)$ are adjacent if and only if the corresponding edges
are adjacent in $G$. The matrix $A(G)=[a_{i,j}]_{n\times n}$ is called the adjacency matrix of a graph $G$ where
\[
a_{i,j}
= \begin{cases}
    1 & \text{if $v_iv_j\in E(G)$,} \\[2mm]
    0 & \text{otherwise,}
  \end{cases}
\]
and $V(G) = \{v_1,v_2,\cdots,v_n\}$. The eigenvalues of a graph $G$ are the eigenvalues of its adjacency matrix $A(G)$. It needs to be mentioned here that all the eigenvalues of a graph $G$ are always real because $A(G)$ is a real symmetric matrix. The multiset of the eigenvalues of a graph $G$ is called the spectrum of $G$ and it is denoted by $Spec(G)$.
If $\lambda_1,\lambda_2,\cdots,\lambda_r$ are distinct eigenvalues of a graph $G$ with multiplicities $m_1,m_2,\cdots,m_r$, respectively, then we write
$$
Spec(G)=  \left(\begin{matrix}
                    \lambda_1     &  \lambda_2 & \cdots & \lambda_r \\[2mm]
                    m_1     &  m_2 & \cdots & m_r
                  \end{matrix}\right).
$$
The spectrum of the union of two graphs is the union of their spectra. Two graphs with the same spectra are known as the cospectral graphs. Graph theoretical terminology not defined here can be found in some standard books of graph theory, like \cite{Bondy-08}.

The energy $\mathcal{E}(G)$ of a graph $G$ is the sum of the absolute values of the eigenvalues of $G$. This concept of graph energy, originating from HMO (H\"{u}ckel molecular orbital) theory, was firstly introduced by Gutman \cite{Gutman-78}. The graph energy can be considered as one of the most studied spectrum-based graph invariants in chemical graph theory. Details about the graph energy can be found in the books \cite{Gutman-16,Li-12}, book chapters \cite{Gutman-01,Gutman-09,Gutman-bc-11}, reviews \cite{Gutman-rev-16,Gutman-rev-17} (it should be mentioned here that scanned copy of the seminal paper \cite{Gutman-78} on graph energy is given in \cite{Gutman-rev-17}), recent papers \cite{Oboudi-19,Ashraf-19,Ashraf-19b,Ma-19,Zhou-19,Aashtab-19,Tian-19,Milovanovic-etal-19,Liu-19,Das-17} and in the related references listed therein.

Two $n$-vertex graphs with the same energy are called equienergetic graphs. Certainly, cospectral graphs are always equienergetic graphs.
Indulal and Vijayakumar \cite{Indulal-08} studied the equienergetic self-complementary graphs and
showed that there exist equienergetic self-complementary graphs on $n$ vertices for every
$n\in \{8,12,16,20,\cdots\}\cup\{73,97,121,145,\cdots\}$. Does there exists some non-self-complementary graph $G$ with the property $\mathcal{E}(G) =
\mathcal{E}( \overline{G} \,)$? The answer of this question is yes:
Ramane \textit{et al.} \cite{Ramane-05} proved that for
an $r$-regular graph $G$, with $r\ge3$, the equation $\mathcal{E}( L(L(G))) =
\mathcal{E}( \overline{L(L(G))}\, )$ holds if and only if $G$ is isomorphic to the complete graph of order 6. This result of Ramane \textit{et al.} gave a birth to the following definition:
two non-self-complementary equienergetic graphs $G_1$ and $G_2$ satisfying $G_1\cong \overline{G_2}$ are called complementary equienergetic graphs.
Recently, Ramane \textit{et al.} \cite{Ramane-19} established
several results concerning the complementary equienergetic regular graphs
%
and suggested to study the complementary equienergetic non-regular graphs.
%
In this paper, we continue the work done in \cite{Ramane-19}. More precisely, we determine all the complementary equienergetic graphs with at most 10 vertices as well as all the members of the graph class $\Omega=\{G \ : \ \mathcal{E}(L(G)) = \mathcal{E}(\overline{L(G)}) \text{, the order of $G$ is at most 10}\}$ by developing some computer codes and by making use of some software.

Firstly, by developing a java code for GraphTea \cite{GraphTea,Rostami-14}, we determined all the graphs $G$ (in graph6 format) of order at most 10 satisfying the equations
\begin{equation}\label{eq-AAA1}
\mathcal{E}(G) = \mathcal{E}(\overline{G})
\end{equation}
 and/or
\begin{equation}\label{eq-AAA2}
 \mathcal{E}(L(G)) = \mathcal{E}(\overline{L(G)})
\end{equation}
up to five decimal places. Then, we made a c-program code and used nauty \cite{NAUTY,McKay-14} for converting the graph6 codes into adjacency matrices.
Afterwards, we filtered these adjacency matrices by verifying Equations \eqref{eq-AAA1} and \eqref{eq-AAA2} up to twelve decimal places -- at this stage we dropped several adjacency matrices because their corresponding graphs do not satisfy either of \eqref{eq-AAA1} and \eqref{eq-AAA2} (for example, there are 54 pairs of graphs satisfying \eqref{eq-AAA1} up to five decimal places but there are only 47 pairs of graphs satisfying \eqref{eq-AAA1} up to 12 decimal places). Finally, we verified Equations \eqref{eq-AAA1} and \eqref{eq-AAA2} for the remaining adjacency matrices by using a high precision computing (200, 500 and then 2000 decimals) of Maple \cite{MAPLE} -- no further adjacency matrix was dropped at this stage.


\section{Main results}


Firstly, we address the problem of determining all the complementary equienergetic graphs on at most 10 vertices -- certainly such graphs occur in pairs.
From the definition of the complementary equienergetic graphs, it follows that if two graphs are complementary equienergetic then they are equienergetic as well, however there exist so many pairs of equienergetic graphs that are not complementary equienergetic. According to Stankovi\'c \textit{et al.} \cite{Stankovic-09} there are 14225 equisets of connected graphs on 10 vertices (a set of graphs is said to be equiset if all graphs in the set have the same energy and there are at least two non-cospectral graphs in the set) but we will see that there are at most 47 pairs of complementary equienergetic graphs on 10 vertices.

It can be easily verified that among all the graphs of order at most 5, there is only one pair of complementary equienergetic graphs and this pair consists of the 4-vertex cycle graph $C_4$ and its complement. Also, there is exactly one pair of complementary equienergetic graphs of order $n$ for every $n\in\{6,7\}$, see Figure \ref{fig1}. Eigenvalues and energies of the pairs of the graphs ($D_1$, $D_2$) and ($E_1$, $E_2$) depicted in Figure \ref{fig1}, are given in Table \ref{Class4}, from which it follows that each of these two pairs is non-cospectral.

\begin{figure}[ht]
 \centering
  \includegraphics[width=0.5\textwidth]{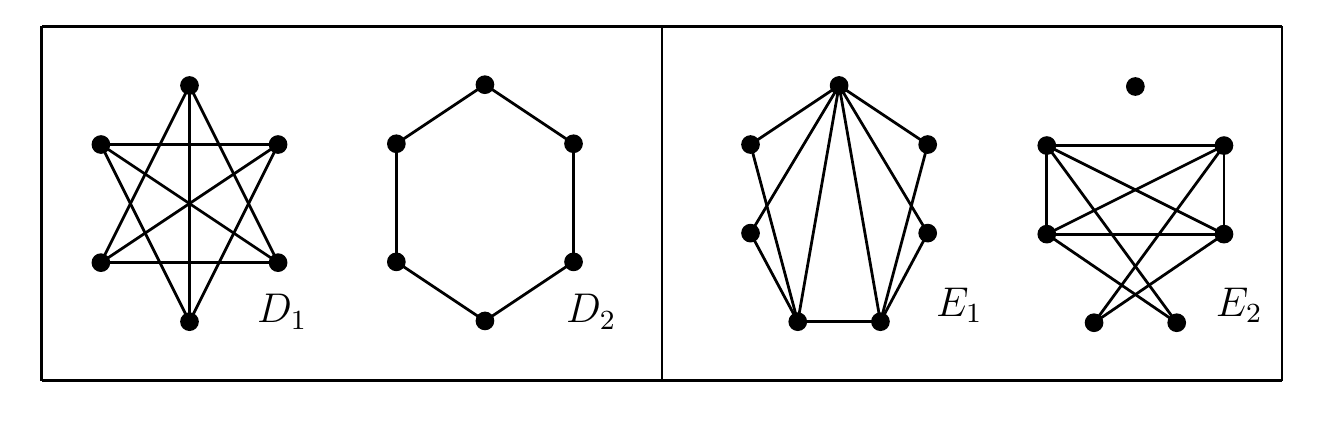}
   \caption{All the pairs of complementary equienergetic graphs on 6 and 7 vertices.}
    \label{fig1}
     \end{figure}

\small
\begin{longtable}[c]{|l|l|l|}
 \caption{Eigenvalues and energies of the graphs shown in Figure \ref{fig1}.\label{Class4}}
\endfirsthead
\multicolumn{3}{c}{Continuation of Table \ref{Class4}}\\ \hline
Graph &  Eigenvalues & Energy \\
\endhead \hline
\endfoot
\hline
Graph &  Eigenvalues & Energy \\ \hline

$D_{1}$ & $-2, -2, 0, 0, 1, 3$ & $8$   \\

$D_{2}$ & $-2, -1, -1, 1, 1, 2$ & \\  \hline

$E_{1}$ & $0, 0, -2, -2, 1, \frac{3}{2}\pm\frac{\sqrt{17}}{2}$ & $5+\sqrt{17}$   \\

$E_{2}$ & $-1, -1, -2,1, \frac{3}{2}\pm\frac{\sqrt{17}}{2}$ & \\
\end{longtable}

\normalsize
\baselineskip=0.30in

There are eight possible pairs of complementary equienergetic graphs of order 8, see Figure \ref{fig2}. From these eight pairs, exactly two pairs ($F_1$, $F_2$) and ($F_{15}$, $F_{16}$) are cospectral. For the remaining six pairs, either their eigenvalues and energies, or just their characteristic polynomials (where
we could not find the closed forms of all the eigenvalues and energies -- in that case energies are verified up to 2000 decimal places using Maple \cite{MAPLE}) are given in Table \ref{Class4a}.

\begin{figure}[ht]
 \centering
  \includegraphics[width=0.75\textwidth]{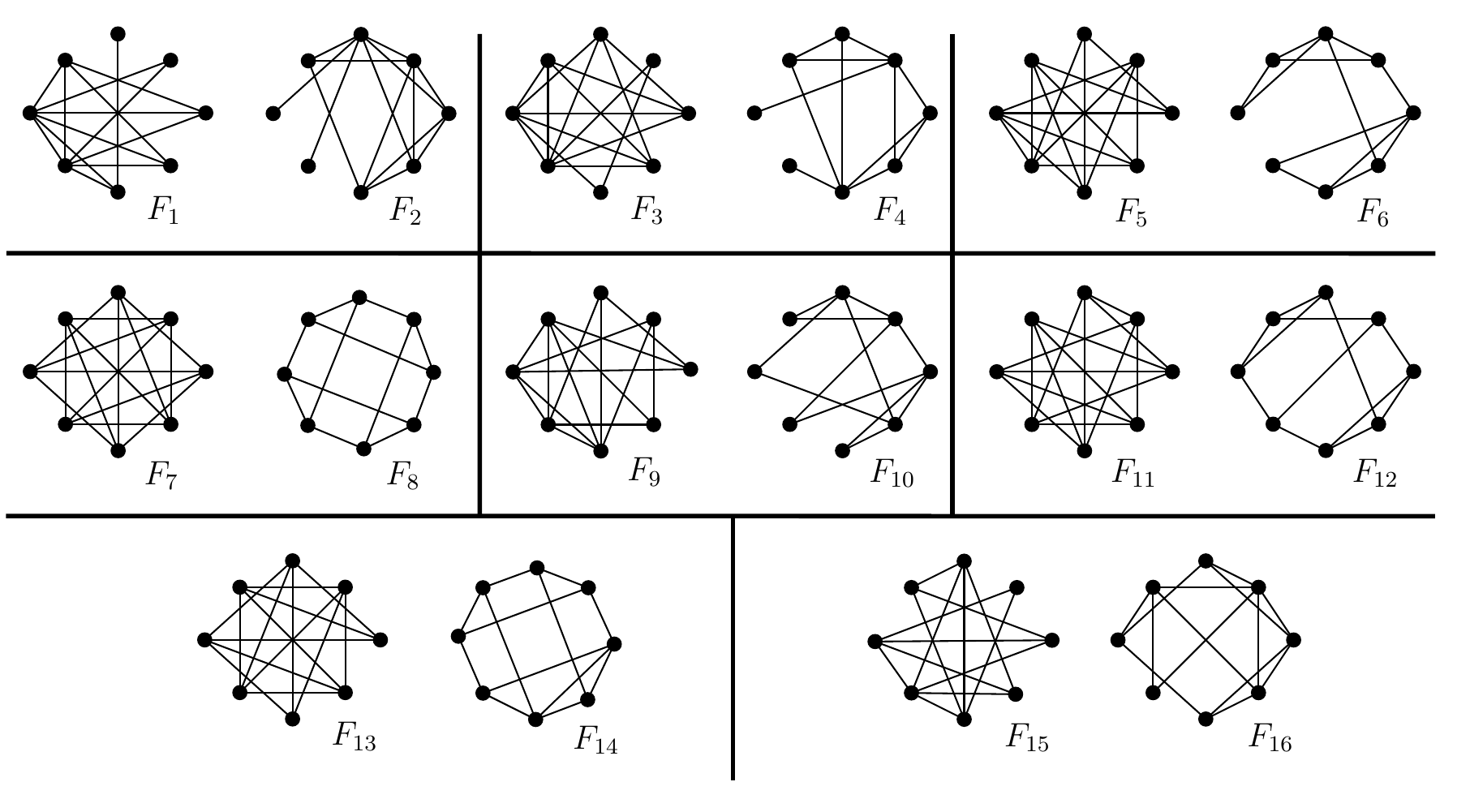}
   \caption{All the possible pairs of complementary equienergetic graphs on 8 vertices.}
    \label{fig2}
     \end{figure}

\small
\begin{longtable}[c]{|l|l|l|}
 \caption{Eigenvalues and energies, or just the characteristic polynomials, of the graphs $F_{3}$, $F_{4}$, $\cdots$, $F_{14}$, depicted in Figure \ref{fig2}.\label{Class4a}}
\endfirsthead
\multicolumn{3}{c}{Continuation of Table \ref{Class4a}}\\ \hline
Graph &  Eigenvalues / Characteristic polynomial & Energy \\
\endhead \hline
\endfoot
\hline
Graph &  Eigenvalues / Characteristic polynomial & Energy \\ \hline

$F_{3}$ & $\lambda^3(\lambda^5 - 16\lambda^3 -24\lambda^2 +16\lambda +32)$ & \\

$F_{4}$ & $\lambda^8 - 12\lambda^6 -8\lambda^5 +22\lambda^4 +16\lambda^3 -12\lambda^2 -8\lambda +1$ & \\  \hline

$F_{5}$ & $\lambda^2(\lambda^6 -16\lambda^4 -12\lambda^3 +28\lambda^2 +16\lambda -12)$ &  \\

$F_{6}$ & $\lambda^8 -12\lambda^6 -8\lambda^5 +34\lambda^4 +40\lambda^3 -8\lambda+1$ &    \\ \hline

$F_{7}$ & $0, 0, 0, -2, -2, -2, 2, 4$ & 12 \\

$F_{8}$ & $3, -3, 1, 1, 1, -1, -1, -1$ &  \\ \hline

$F_{9}$ & $-1, 1, 2\pm\sqrt{5}, -1\pm\sqrt{2}, -1\pm\sqrt{2}$ & $2(1+\sqrt{5}+2\sqrt{2})$ \\

$F_{10}$ & $0,  \pm\sqrt{2}, \pm\sqrt{2}, -2, 1\pm\sqrt{5}$ &  \\ \hline

$F_{11}$ & $0, 0, -2, 4, \pm\sqrt{2}, -1\pm\sqrt{3}$ & $2(3+\sqrt{2}+\sqrt{3})$   \\

$F_{12}$ & $-1, -1, \pm\sqrt{3}, -1\pm\sqrt{2}, 3, 1$ &  \\ \hline

$F_{13}$ & $\lambda(\lambda^7-15\lambda^5-12\lambda^4+45\lambda^3+40\lambda^2-24\lambda-8)$ &  \\

$F_{14}$ & $\lambda^8-13\lambda^6-4\lambda^5+36\lambda^4+8\lambda^3-33\lambda^2-4\lambda+9$ &  \\ \hline

\end{longtable}

\normalsize
\baselineskip=0.30in

On 9 vertices, there are exactly twenty four pairs of complementary equienergetic graphs, which are depicted in Figure \ref{fig3}. From these twenty four pairs, exactly nineteen pairs, namely ($G_{2i-1}$, $G_{2i}$) with $i=6,7,\cdots,24$, are cospectral. For the remaining five pairs, their eigenvalues and energies are given in Table \ref{Class4b}.

\begin{figure}
 \centering
  \includegraphics[width=0.99\textwidth]{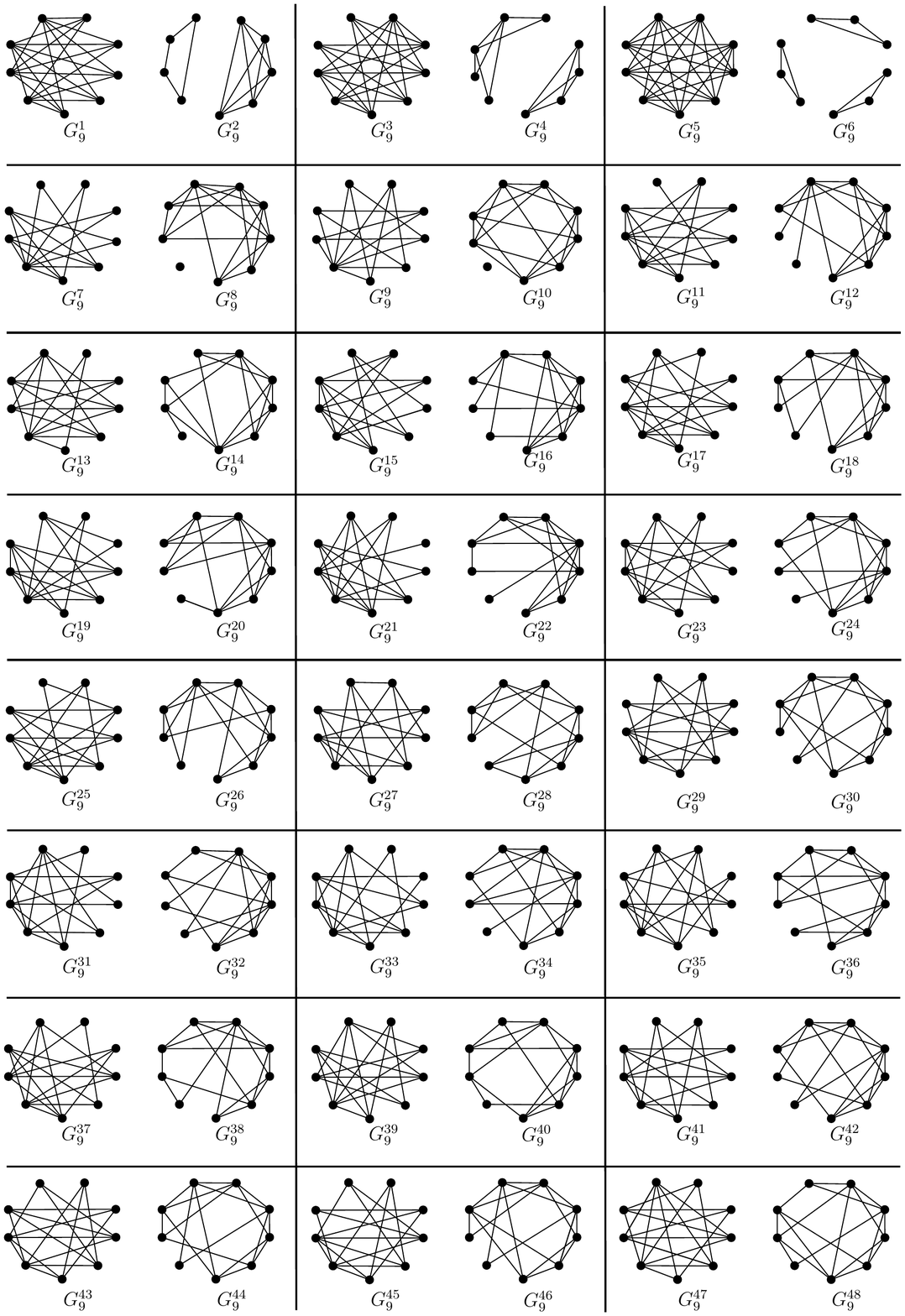}
   \caption{All the pairs of complementary equienergetic graphs of order 9.}
    \label{fig3}
     \end{figure}

\small
\begin{longtable}[c]{|l|l|l|}
 \caption{Eigenvalues and energies of the graphs $G_{1}$, $G_{2}$, $\cdots$, $G_{10}$, shown in Figure \ref{fig3}.\label{Class4b}}
\endfirsthead
\multicolumn{3}{c}{Continuation of Table \ref{Class4b}}\\ \hline
Graph &  Eigenvalues & Energy \\
\endhead \hline
\endfoot
\hline
Graph &  Eigenvalues & Energy \\ \hline

$G_{1}$ & $0,0,0,0,-1,-1,1,5,-4$ & 12 \\

$G_{2}$ & $0,0,-1,-1,-1,-1,2,4,-2$ & \\ \hline

$G_{3}$ & $0, 0, 0, 0, -1, -1, -4, 3\pm \sqrt{5}$ & 12 \\

$G_{4}$ & $0, 0, -1, -1, -1, -1, 3, 3, -2$ &    \\ \hline

$G_{5}$ & $-3, -3, 0, 0, 0, 0, 0, 0, 0, 6$ & 12 \\

$G_{6}$ & $-1, -1, -1, -1, -1, -1, 2, 2, 2$ & \\ \hline

$G_{7}$ & $0, -1, -1, 1, 1, -1\pm\sqrt{2}, 1\pm \sqrt{10}$ & $2(\sqrt{10}+\sqrt{2}+2)$ \\

$G_{8}$ & $0, 0, 0, -2, -2, \pm\sqrt{2}, 2 \pm \sqrt{10} $ & \\ \hline

$G_{9}$ & $0, 0, -2, -2,\pm\sqrt{2}, \pm\sqrt{2}, 4$ & $4(2+\sqrt{2})$ \\

$G_{10}$ & $0, -1, -1, -1\pm\sqrt{2}, -1\pm\sqrt{2}, 5, 1 $ & \\ \hline

\end{longtable}

\normalsize
\baselineskip=0.30in

All the possible pairs of complementary equienergetic graphs of order 10 (which are forty seven in total) are shown in Figures \ref{fig4} and \ref{fig5}. We note that none of these pairs is cospectral. We give either the eigenvalues and energies, or just the characteristic polynomials (where
we could not find the closed forms of the eigenvalues and energies -- in that case energies are verified up to 2000 decimal places using Maple \cite{MAPLE}) in Table \ref{Class5}.

\begin{figure}
 \centering
  \includegraphics[width=0.85\textwidth]{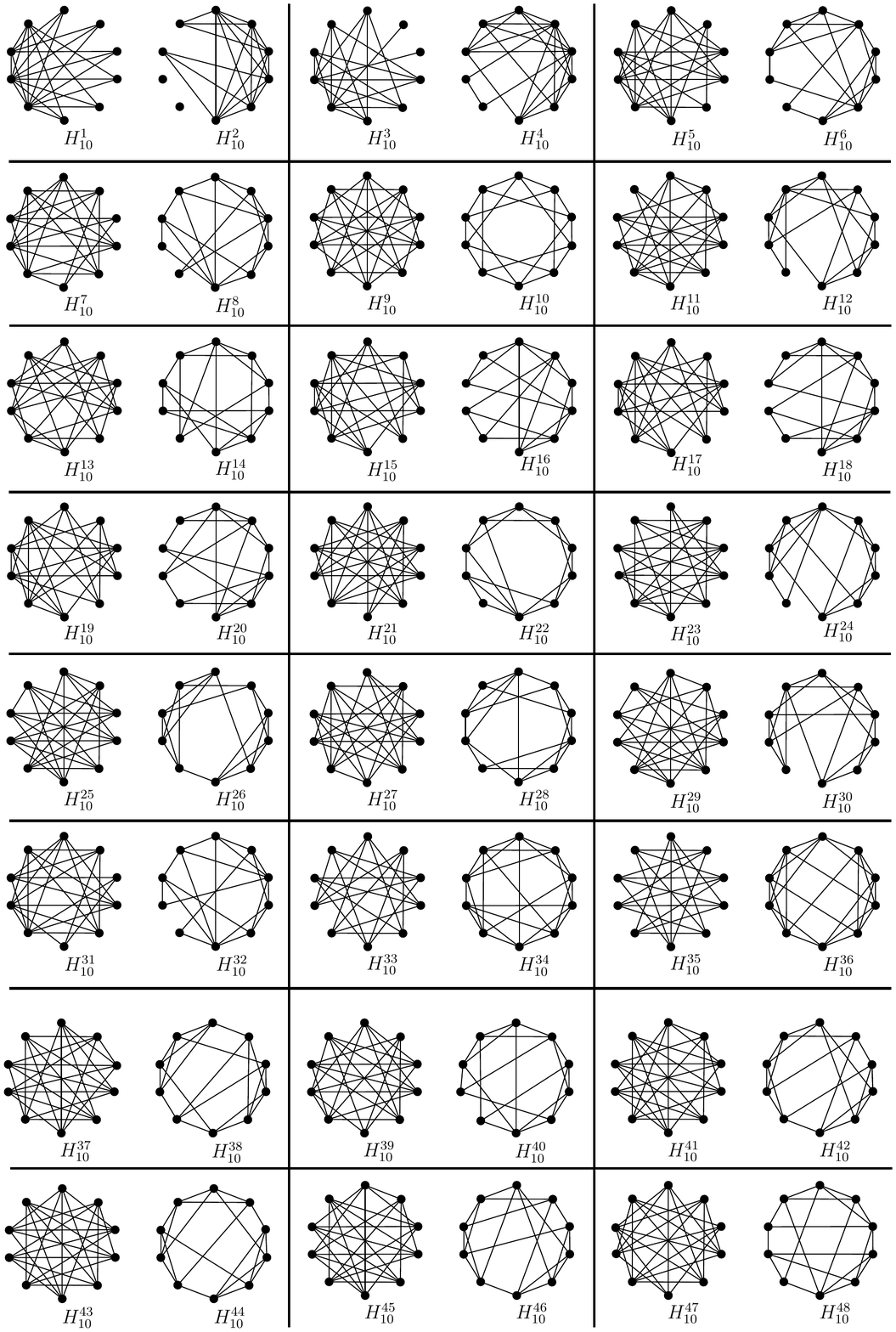}
   \caption{Twenty four possible pairs of complementary equienergetic graphs on 10 vertices. The remaining twenty three possible pairs of complementary equienergetic graphs on 10 vertices are shown in Figure \ref{fig5}.}
    \label{fig4}
      \end{figure}

\begin{figure}
 \centering
  \includegraphics[width=0.9\textwidth]{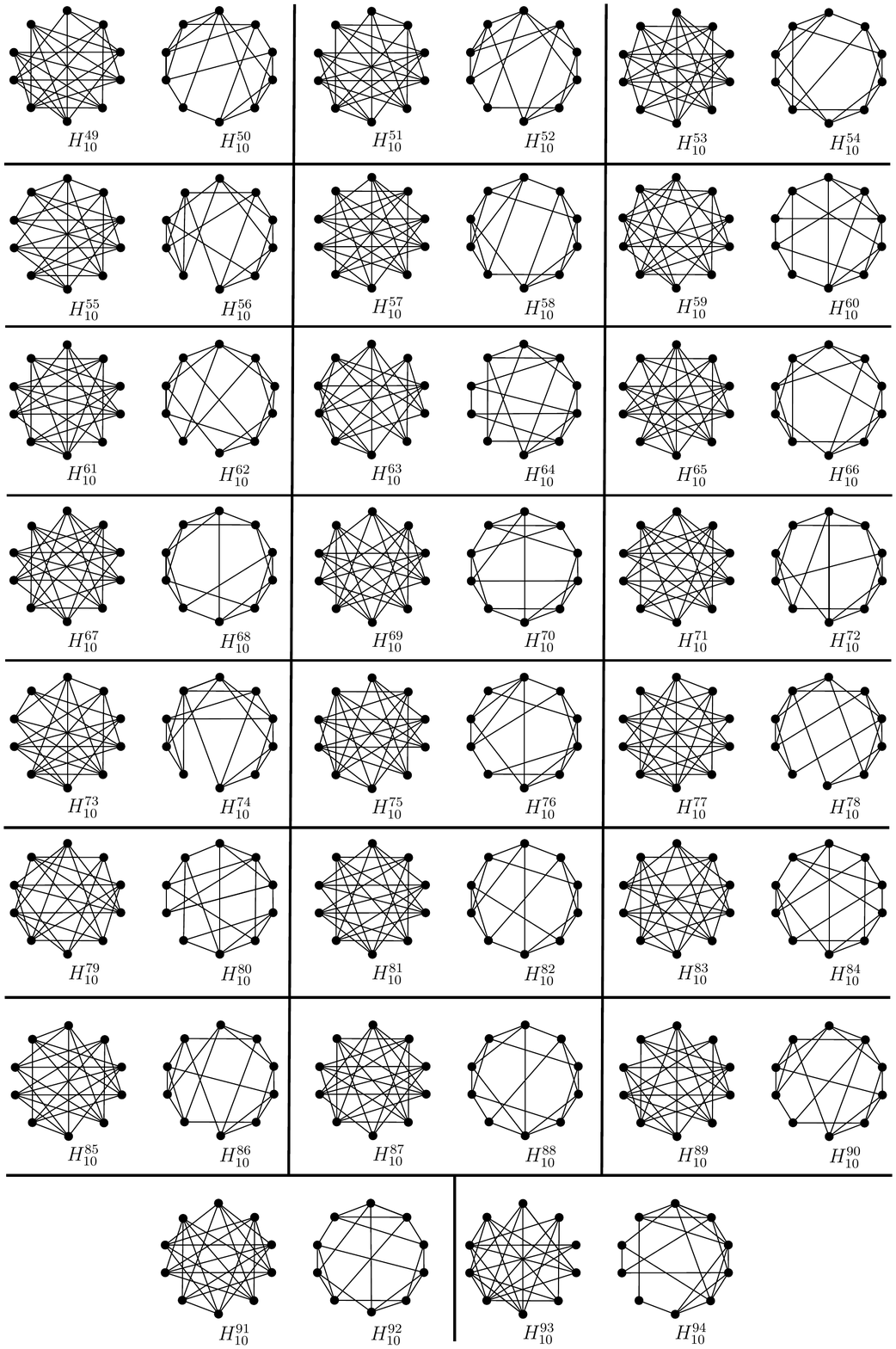}
   \caption{Twenty three possible pairs of complementary equienergetic graphs on 10 vertices, different from those given in Figure \ref{fig4}.}
    \label{fig5}
      \end{figure}

\footnotesize
\begin{longtable}[c]{|l|l|l|}
\caption{Eigenvalues and energies, or just the characteristic polynomials of the graphs depicted in Figures \ref{fig4} and \ref{fig5}.\label{Class5}}
\endfirsthead
\multicolumn{3}{c}{Continuation of Table \ref{Class5}}\\ \hline
Graph &  Eigenvalues / Characteristic polynomial & Energy \\
\hline
\endhead \hline
\endfoot
\hline
Graph &  Eigenvalues / Characteristic polynomial & Energy \\ \hline
$H_{1}$ &  $ 0, 0, 0, 0, -3, -1, \frac{5}{2}\pm \frac{\sqrt{37}}{2}, -\frac{1}{2} \pm \frac{\sqrt{13}}{2}$ & $4+\sqrt{37}+\sqrt{13}$ \\

$H_{2}$ &  $0, 0, \frac{5}{2} \pm \frac{\sqrt{37}}{2}, -\frac{1}{2}\pm \frac{\sqrt{13}}{2}, -1, -1, -1, -1$ &  \\ \hline

$H_{3}$ & $ ({\lambda}+2)({\lambda}^3-2{\lambda}^2-13{\lambda}+2)({\lambda}-1)^3({\lambda}+1)^3$ & \\

$H_{4}$ & $ {\lambda}^3({\lambda}-1)({\lambda}^3-5{\lambda}^2-6{\lambda}+12)({\lambda}+2)^3$ & \\ \hline

$H_{5}$ & $ {\lambda}^3({\lambda}+2)({\lambda}+1)({\lambda}^2+2{\lambda}-4)({\lambda}^3-5{\lambda}^2-4{\lambda}+12)$ & \\

$H_{6}$ & $ {\lambda}({\lambda}-1)({\lambda}^2-5)({\lambda}^3-2{\lambda}^2-11{\lambda}+4)({\lambda}+1)^3$ & \\ \hline

$H_{7}$ & $ {\lambda}^3({\lambda}+2)({\lambda}+1)({\lambda}^2+2{\lambda}-4)({\lambda}^3-5{\lambda}^2-4{\lambda}+12)$ &\\

$H_{8}$ & $ {\lambda}({\lambda}-1)({\lambda}^2-5)({\lambda}^3-2{\lambda}^2-11{\lambda}+4)({\lambda}+1)^3$ & \\ \hline

$H_{9}$ & $-2, -2, -2, -2, 0, 0, 0, 0, 3, 5$  & $ 16$  \\

$H_{10}$ & $-4, -1, -1, -1, -1, 1, 1, 1, 1, 4$  &   \\ \hline

$H_{11}$ & $ {\lambda}^3({\lambda}-1)({\lambda}+2)({\lambda}^2-4{\lambda}-6)({\lambda}^3+3{\lambda}^2-4{\lambda}-10)$ &\\

$H_{12}$ & $ {\lambda}({\lambda}-1)({\lambda}^2-2{\lambda}-9)({\lambda}^3-7{\lambda}+4)({\lambda}+1)^3$ &\\ \hline

$H_{13}$ & $ {\lambda}^3({\lambda}+2)({\lambda}^3-5{\lambda}^2-2{\lambda}+8)({\lambda}^3+3{\lambda}^2-4{\lambda}-8)$ &\\

$H_{14}$ & $ ({\lambda}-1)({\lambda}^3-7{\lambda}+2)({\lambda}^3-2{\lambda}^2-9{\lambda}+2)({\lambda}+1)^3$ &\\  \hline

$H_{15}$ & $ {\lambda}^3({\lambda}+2)({\lambda}+1)({\lambda}^2+2{\lambda}-4)({\lambda}^3-5{\lambda}^2-4{\lambda}+16)$ &\\

$H_{16}$ & $ {\lambda}({\lambda}-1)({\lambda}^2-5)({\lambda}^3-2{\lambda}^2-11{\lambda}+8)({\lambda}+1)^3$ &\\  \hline

$H_{17}$ & $ {\lambda}^3({\lambda}+2)({\lambda}+1)({\lambda}^2+2{\lambda}-4)({\lambda}^3-5{\lambda}^2-4{\lambda}+16)$ &\\

$H_{18}$ & $ {\lambda}({\lambda}-1)({\lambda}^2-5)({\lambda}^3-2{\lambda}^2-11{\lambda}+8)({\lambda}+1)^3$ &\\  \hline

$H_{19}$ & ${\lambda}^3({\lambda}+2)({\lambda}+1)({\lambda}^2+2{\lambda}-4)({\lambda}^3-5{\lambda}^2-4{\lambda}+16)$ &\\

$H_{20}$ & $ {\lambda}({\lambda}-1)({\lambda}^2-5)({\lambda}^3-2{\lambda}^2-11{\lambda}+8)({\lambda}+1)^3$ &\\ \hline

$H_{21}$ & $ -1, 1, 0, 0, 2 \pm \sqrt{10}, -1 \pm \sqrt{5},   -1\pm\sqrt{3}$ & $2(1+\sqrt{10}+\sqrt{5}+\sqrt{3})$ \\

$H_{22}$ & $-1, -1, 0, 0, \pm\sqrt{5}, \pm\sqrt{3}, 1\pm\sqrt{10}$ &  \\   \hline

$H_{23}$ & $-2, 1, 0, 0, \pm \sqrt{2}, 2 \pm \sqrt{10},  -\frac{3}{2}\pm\frac{\sqrt{17}}{2}$ &     $3+2\sqrt{2}+2\sqrt{10}+\sqrt{17}$ \\

$H_{24} $& $ 0, 1, -1, -1, 1\pm\sqrt{10}, \frac{1}{2}\pm\frac{\sqrt{17}}{2}, -1\pm\sqrt{2}$ &    \\ \hline

$H_{25}$ & $ {\lambda}^2({\lambda}-1)({\lambda}-5)({\lambda}+1)({\lambda}^2+{\lambda}-1)({\lambda}^3+4{\lambda}^2-2{\lambda}-11)$ & \\

$H_{26}$ & $ {\lambda}({\lambda}-4)({\lambda}+2)({\lambda}^2+{\lambda}-1)({\lambda}^3-{\lambda}^2-7{\lambda}+6)({\lambda}+1)^2$ &\\  \hline

$H_{27}$ & $ {\lambda}^2({\lambda}^3-5{\lambda}^2-2{\lambda}+8)({\lambda}^5+5{\lambda}^4+2{\lambda}^3-16{\lambda}^2-12{\lambda}+4)$ & \\

$H_{28}$ & $ ({\lambda}^3-2{\lambda}^2-9{\lambda}+2)({\lambda}^5-8{\lambda}^3+2{\lambda}^2+11{\lambda}-2)({\lambda}+1)^2$ & \\  \hline

$H_{29}$ & $ {\lambda}^2({\lambda}^3-5{\lambda}^2-2{\lambda}+8)({\lambda}^5+5{\lambda}^4+2{\lambda}^3-16{\lambda}^2-12{\lambda}+4)$ & \\

$H_{30}$ & $ ({\lambda}^3-2{\lambda}^2-9{\lambda}+2)({\lambda}^5-8{\lambda}^3+2{\lambda}^2+11{\lambda}-2)({\lambda}+1)^2$ & \\  \hline

$H_{31}$ & $ {\lambda}^3({\lambda}^2+{\lambda}-4)({\lambda}^3-5{\lambda}^2-4{\lambda}+12)({\lambda}+2)^2$ & \\

$H_{32}$ & $ ({\lambda}^2+{\lambda}-4)({\lambda}^3-2{\lambda}^2-11{\lambda}+4)({\lambda}-1)^2({\lambda}+1)^3$ &\\ \hline

$H_{33}$ &  $ -3, 0, 4, 1, 1, -1, -1, -1, \pm \sqrt{5} $ & $12+2\sqrt{5}$ \\

$H_{34}$ & $-1, 2, 5, -2, -2, 0, 0, 0, -1 \pm \sqrt{5}$ &  \\ \hline

$H_{35}$ & $ ({\lambda}-4)({\lambda}^4+3{\lambda}^3-5{\lambda}^2-11{\lambda}+4)({\lambda}-1)^2({\lambda}+1)^3$ &\\

$H_{36}$ & $ {\lambda}^3({\lambda}-5)({\lambda}^4+{\lambda}^3-8{\lambda}^2-4{\lambda}+8)({\lambda}+2)^2$ &\\  \hline

$H_{37}$ & $ {\lambda}^2({\lambda}-5)({\lambda}^3+3{\lambda}^2-5{\lambda}-11)({\lambda}^2+{\lambda}-1)^2$ &\\

$H_{38}$ & $ ({\lambda}-4)({\lambda}^3-8{\lambda}+4)({\lambda}+1)^2({\lambda}^2+{\lambda}-1)^2$ &\\  \hline

$H_{39}$ & $ {\lambda}^3({\lambda}-2)({\lambda}+3)({\lambda}^3-5{\lambda}^2-2{\lambda}+8)({\lambda}+2)^2$ &\\

$H_{40}$ & $({\lambda}-2)({\lambda}+3)({\lambda}^3-2{\lambda}^2-9{\lambda}+2)({\lambda}-1)^2({\lambda}+1)^3$ &\\  \hline

$H_{41}$ & $ {\lambda}^2({\lambda}-5)({\lambda}+1)({\lambda}^2+{\lambda}-1)({\lambda}^4+3{\lambda}^3-6{\lambda}^2-13{\lambda}+11)$ &\\

$H_{42}$ & $ {\lambda}({\lambda}-4)({\lambda}^2+{\lambda}-1)({\lambda}^4+{\lambda}^3-9{\lambda}^2-4{\lambda}+16)({\lambda}+1)^2$ & \\  \hline

$H_{43}$ & $ {\lambda}^2({\lambda}-5)({\lambda}+1)({\lambda}^3-4{\lambda}+2)({\lambda}^3+4{\lambda}^2-6)$ & \\

$H_{44}$ & $ {\lambda}({\lambda}-4)({\lambda}^3-{\lambda}^2-5{\lambda}+3)({\lambda}^3+3{\lambda}^2-{\lambda}-5)({\lambda}+1)^2$ &\\  \hline

$H_{45}$ & $ {\lambda}^3({\lambda}-5)({\lambda}+2)({\lambda}^2-2)({\lambda}^3+3{\lambda}^2-4{\lambda}-10)$ &\\

$H_{46}$ & $ ({\lambda}-1)({\lambda}-4)({\lambda}^2+2{\lambda}-1)({\lambda}^3-7{\lambda}+4)({\lambda}+1)^3$ &\\  \hline

$H_{47}$ & $ {\lambda}^3({\lambda}+3)({\lambda}+2)({\lambda}^2-2)({\lambda}^3-5{\lambda}^2-4{\lambda}+18)$ &\\

$H_{48}$ & $ ({\lambda}-1)({\lambda}-2)({\lambda}^2+2{\lambda}-1)({\lambda}^3-2{\lambda}^2-11{\lambda}+10)({\lambda}+1)^3$  & \\  \hline

$H_{49}$ & $ {\lambda}^3({\lambda}+2)({\lambda}^3+3{\lambda}^2-3{\lambda}-8)({\lambda}^3-5{\lambda}^2-3{\lambda}+14)$ & \\

$H_{50}$ & $ ({\lambda}-1)({\lambda}^3-2{\lambda}^2-10{\lambda}+7)({\lambda}^3-6{\lambda}+3)({\lambda}+1)^3$ &\\  \hline

$H_{51}$ & $ {\lambda}^3({\lambda}+2)({\lambda}^3+3{\lambda}^2-3{\lambda}-8)({\lambda}^3-5{\lambda}^2-3{\lambda}+14)$ &\\

$H_{52}$ & $ ({\lambda}-1)({\lambda}^3-2{\lambda}^2-10{\lambda}+7)({\lambda}^3-6{\lambda}+3)({\lambda}+1)^3$ &\\  \hline

$H_{53}$ & $ {\lambda}^2({\lambda}-5)({\lambda}^3+4{\lambda}^2+2{\lambda}-2)({\lambda}^4+{\lambda}^3-6{\lambda}^2-4{\lambda}+6)$ &\\

$H_{54}$ & $ ({\lambda}-4)({\lambda}^3-{\lambda}^2-3{\lambda}+1)({\lambda}^4+3{\lambda}^3-3{\lambda}^2-7{\lambda}+4)({\lambda}+1)^2$ & \\ \hline

$H_{55}$ & $-1, 5, 0, 0, \pm \sqrt{2}, -1 \pm\sqrt{5},-1\pm\sqrt{3}$ & $ 2(3+\sqrt{2}+\sqrt{5}+\sqrt{3})$ \\

$H_{56}$ & $ 0, 4, -1, -1, \pm \sqrt{5}, \pm \sqrt{3},  -1\pm\sqrt{2}$ &  \\  \hline

$H_{57}$ & $-2, 5, 0, 0, \pm \sqrt{2}, \pm \sqrt{2}, -\frac{3}{2} \pm \frac{\sqrt{17}}{2}$  & $ 7+ \sqrt{17}+4\sqrt{2}$ \\

$H_{58}$ & $ 1, 4, -1, -1, \frac{1}{2} \pm \frac{\sqrt{17}}{2},  -1\pm\sqrt{2},  -1\pm\sqrt{2}$ &   \\ \hline

$H_{59}$ & ${\lambda}^3({\lambda}-1)({\lambda}^3-5{\lambda}^2-6{\lambda}+28)({\lambda}+2)^3$ &\\

$H_{60}$ & $({\lambda}+2)({\lambda}^3-2{\lambda}^2-13{\lambda}+18)({\lambda}-1)^3({\lambda}+1)^3$ &\\  \hline

$H_{61}$ & ${\lambda}^2({\lambda}+3)({\lambda}^2-2)({\lambda}^2+2{\lambda}-2)({\lambda}^3-5{\lambda}^2-2{\lambda}+8)$ &\\

$H_{62}$ & $ ({\lambda}-2)({\lambda}^2-3)({\lambda}^2+2{\lambda}-1)({\lambda}^3-2{\lambda}^2-9{\lambda}+2)({\lambda}+1)^2$ &\\  \hline

$H_{63}$ & $ {\lambda}^2({\lambda}+2)({\lambda}^2+2{\lambda}-4)({\lambda}^2+{\lambda}-1)({\lambda}^3-5{\lambda}^2-3{\lambda}+12)$ &\\

$H_{64}$ & $ ({\lambda}-1)({\lambda}^2-5)({\lambda}^2+{\lambda}-1)({\lambda}^3-2{\lambda}^2-10{\lambda}+5)({\lambda}+1)^2$ &\\  \hline

$H_{65}$ & $ {\lambda}^2({\lambda}-5)({\lambda}^3+2{\lambda}^2-4{\lambda}-6)({\lambda}^4+3{\lambda}^3-2{\lambda}^2-6{\lambda}+2)$ &\\

$H_{66}$ & $ ({\lambda}-4)({\lambda}^3+{\lambda}^2-5{\lambda}+1)({\lambda}^4+{\lambda}^3-5{\lambda}^2-3{\lambda}+4)({\lambda}+1)^2$ &\\  \hline

$H_{67}$ & $ {\lambda}^2({\lambda}-5)({\lambda}^2+{\lambda}-1)({\lambda}^5+4{\lambda}^4-3{\lambda}^3-19{\lambda}^2+2{\lambda}+19)$ &\\

$H_{68}$ & $ ({\lambda}-4)({\lambda}^2+{\lambda}-1)({\lambda}^5+{\lambda}^4-9{\lambda}^3-4{\lambda}^2+20{\lambda}-4)({\lambda}+1)^2$ &\\  \hline

$H_{69}$ & $ {\lambda}^2({\lambda}-5)({\lambda}^2+{\lambda}-1)({\lambda}^5+4{\lambda}^4-3{\lambda}^3-19{\lambda}^2+2{\lambda}+19)$ &\\

$H_{70}$ & $ ({\lambda}-4)({\lambda}^2+{\lambda}-1)({\lambda}^5+{\lambda}^4-9{\lambda}^3-4{\lambda}^2+20{\lambda}-4)({\lambda}+1)^2$ &\\  \hline

$H_{71}$ & $ {\lambda}^2({\lambda}-5)({\lambda}+2)({\lambda}^6+3{\lambda}^5-6{\lambda}^4-16{\lambda}^3+12{\lambda}^2+16{\lambda}-8)$ &\\

$H_{72}$ & $ ({\lambda}-1)({\lambda}-4)({\lambda}^6+3{\lambda}^5-6{\lambda}^4-18{\lambda}^3+9{\lambda}^2+23{\lambda}-4)({\lambda}+1)^2$ &\\  \hline

$H_{73}$ & $ {\lambda}^2({\lambda}-1)({\lambda}^2+2{\lambda}-4)({\lambda}^3-5{\lambda}^2-2{\lambda}+8)({\lambda}+2)^2$ & \\

$H_{74}$ & $ ({\lambda}+2)({\lambda}^2-5)({\lambda}^3-2{\lambda}^2-9{\lambda}+2)({\lambda}-1)^2({\lambda}+1)^2$ & \\  \hline

$H_{75}$ & $ {\lambda}^2({\lambda}-1)({\lambda}^2+2{\lambda}-4)({\lambda}^3-5{\lambda}^2-2{\lambda}+8)({\lambda}+2)^2$ & \\

$H_{76}$ & $ ({\lambda}+2)({\lambda}^2-5)({\lambda}^3-2{\lambda}^2-9{\lambda}+2)({\lambda}-1)^2({\lambda}+1)^2$ & \\  \hline

$H_{77}$ & $ {\lambda}^2({\lambda}-1)({\lambda}^2+2{\lambda}-4)({\lambda}^3-5{\lambda}^2-2{\lambda}+8)({\lambda}+2)^2$ &\\

$H_{78}$ & $ ({\lambda}+2)({\lambda}^2-5)({\lambda}^3-2{\lambda}^2-9{\lambda}+2)({\lambda}-1)^2({\lambda}+1)^2$ & \\  \hline

$H_{79}$ & $ {\lambda}^2({\lambda}-1)({\lambda}^2+2{\lambda}-4)({\lambda}^3-5{\lambda}^2-2{\lambda}+8)({\lambda}+2)^2$ &\\

$H_{80}$ & $ ({\lambda}+2)({\lambda}^2-5)({\lambda}^3-2{\lambda}^2-9{\lambda}+2)({\lambda}-1)^2({\lambda}+1)^2$ &\\  \hline

$H_{81}$ & $ {\lambda}^2({\lambda}-5)({\lambda}^2-2)({\lambda}^3+{\lambda}^2-6{\lambda}+2)({\lambda}+2)^2$ & \\

$H_{82}$ & $ ({\lambda}-4)({\lambda}^2+2{\lambda}-1)({\lambda}^3+2{\lambda}^2-5{\lambda}-8)({\lambda}-1)^2({\lambda}+1)^2$ &  \\ \hline

$H_{83} $  & $  2, 5, -2, -2, -2, 0, 0, 0,   -\frac{1}{2}\pm\frac{\sqrt{17}}{2}$  & $ 13+\sqrt{17}$  \\

$H_{84}$ & $-3, 4, -1, -1, -1, 1, 1, 1, -\frac{1}{2}\pm\frac{\sqrt{17}}{2}$ &   \\  \hline

$H_{85}$ &  $  -3, -1, 1, 5, 0, 0, -\frac{1}{2}\pm\frac{\sqrt{13}}{2}, -\frac{1}{2}\pm\frac{\sqrt{13}}{2}$   & $ 10+2\sqrt{13}$   \\

$H_{86}$ & $ -2, 0, 2, 4, -1, -1, -\frac{1}{2}\pm\frac{\sqrt{13}}{2},-\frac{1}{2}\pm\frac{\sqrt{13}}{2}$  &    \\ \hline

$H_{87}$ & $ {\lambda}^2({\lambda}^3+{\lambda}^2-5{\lambda}+2)({\lambda}^3-5{\lambda}^2-3{\lambda}+14)({\lambda}+2)^2$ & \\

$H_{88}$ & $ ({\lambda}^3-2{\lambda}^2-10{\lambda}+7)({\lambda}^3+2{\lambda}^2-4{\lambda}-7)({\lambda}-1)^2({\lambda}+1)^2$ & \\  \hline

$H_{89}$ & $ {\lambda}^2({\lambda}^3+{\lambda}^2-5{\lambda}+2)({\lambda}^3-5{\lambda}^2-3{\lambda}+14)({\lambda}+2)^2$ & \\

$H_{90}$ & $ ({\lambda}^3-2{\lambda}^2-10{\lambda}+7)({\lambda}^3+2{\lambda}^2-4{\lambda}-7)({\lambda}-1)^2({\lambda}+1)^2$ & \\ \hline

$H_{91}$ &  $ -2, 5, 0, 0, \pm\sqrt{2},  -\frac{1}{2}\pm\frac{\sqrt{17}}{2}, -1\pm\sqrt{3} $ & $ 7+2\sqrt{3}+\sqrt{17}+2\sqrt{2}$ \\

$H_{92}$ &   $  1, 4, -1, -1, \pm\sqrt{3}, -\frac{1}{2}\pm\frac{\sqrt{17}}{2}, -1\pm\sqrt{2} $ &  \\  \hline

$H_{93}$ & $ 5, -2, -2, 1, 1, 1, -1\pm\sqrt{2},  -1\pm\sqrt{2}$ & $  12+4\sqrt{2}$ \\

$H_{94}$ & $  1, -2, -2, -2, \pm \sqrt{2}, \pm\sqrt{2},  \frac{5}{2}\pm\frac{\sqrt{17}}{2}$ &  \\  \hline

\end{longtable}

\normalsize
\baselineskip=0.30in

Next, we determine all the possible members of the graph class $$\Omega=\{G \ : \ \mathcal{E}(L(G)) = \mathcal{E}(\overline{L(G)}) \text{, the order of $G$ is at most 10}\},$$ see Figure \ref{fig6}. In what follows, we prove a result, which guaranties that the graph $LG^{9}_5 $ (shown in Figure \ref{fig6}) belongs to the class $ \Omega$.

\begin{proposition}\label{prop-A0}

For $p\ge4$, it holds that $\mathcal{E}\left(L\left(K_{p+2}^{(p)}\right)\right) = \mathcal{E}\left(\overline{L\left(K_{p+2}^{(p)}\right)}\right)$, where $K_p^{(q)}$ is the graph obtained from the complete graphs $K_p$ and $K_q$ by identifying one of their vertices.

\end{proposition}

\begin{proof}
The result follows from the following facts
$$
Spec\left(L\left(K_{p+2}^{(p)}\right)\right)=
  \left(\begin{matrix}
                    -2     & p-3 & p-1 & \frac{4p-1-\sqrt{8p+25}}{2} & 2p-1 & \frac{4p-1+\sqrt{8p+25}}{2} \\[3mm]
                    p(p+1) & p-1 & p   &  1                          &  1   & 1
                  \end{matrix}\right)
$$
and
$$
Spec\left(\overline{L\left(K_{p+2}^{(p)}\right)}\right)=
  \left(\begin{matrix}
                    1     & -p+2 & -p    & -2(p-1) & -2p & p(p+1) \\[3mm]
                    p(p+1)&  p-1 & p+1   &  1      &  1  & 1
                  \end{matrix}\right).
$$
\end{proof}

It is clear that $LG^{9}_5 \cong K_6^{(4)}$ and hence by using Proposition \ref{prop-A0}, we deduce that $LG^{9}_5 \in \Omega$. Also, the next proposition (Proposition \ref{prop-A1}) ensures that the graphs
$LG^{5}_1$, $LG^{6}_2$, $LG^{7}_1$, $LG^{8}_2$, $LG^{8}_3$, $LG^{8}_5$, $LG^{9}_1$, $LG^{9}_3$, $LG^{9}_4$, $LG^{10}_1$, $LG^{10}_2$, $LG^{10}_4$ and $LG^{10}_8$ belong to the class $\Omega$.

\begin{proposition} {\rm \cite{Ramane-19}} \label{prop-A1}
For $p,q\ge2$, it holds that $\mathcal{E}\left(L\left(K_{p,q}\right)\right) = \mathcal{E}\left(\overline{L\left(K_{p,q}\right)}\right)$, where $K_{p,q}$ is the complete bipartite graph of order $p+q$.

\end{proposition}

The energies or the characteristic polynomials of the line graphs and complements of the line graphs, of those graphs of Figure \ref{fig6} which are different from $LG^{5}_1$, $LG^{6}_2$, $LG^{7}_1$, $LG^{7}_3$, $LG^{8}_2$, $LG^{8}_3$, $LG^{8}_5$, $LG^{9}_1$, $LG^{9}_3$, $LG^{9}_4$, $LG^{10}_1$, $LG^{10}_2$, $LG^{10}_4$ and $LG^{10}_8$, are specified in Table \ref{Class6}.
We remark that the equations
$ \mathcal{E}\left(L(LG^{10}_9)\right) = \mathcal{E}\left(\overline{L(LG^{10}_9)}\right)$,
$ \mathcal{E}\left(L(LG^{10}_{10})\right) = \mathcal{E}\left(\overline{L(LG^{10}_{10})}\right)$
and
$ \mathcal{E}\left(L(LG^{10}_{12})\right) = \mathcal{E}\left(\overline{L(LG^{10}_{12})}\right)$
are verified up to 2000 decimal places using Maple \cite{MAPLE}.

\begin{figure}[ht]
 \centering
  \includegraphics[width=0.752\textwidth]{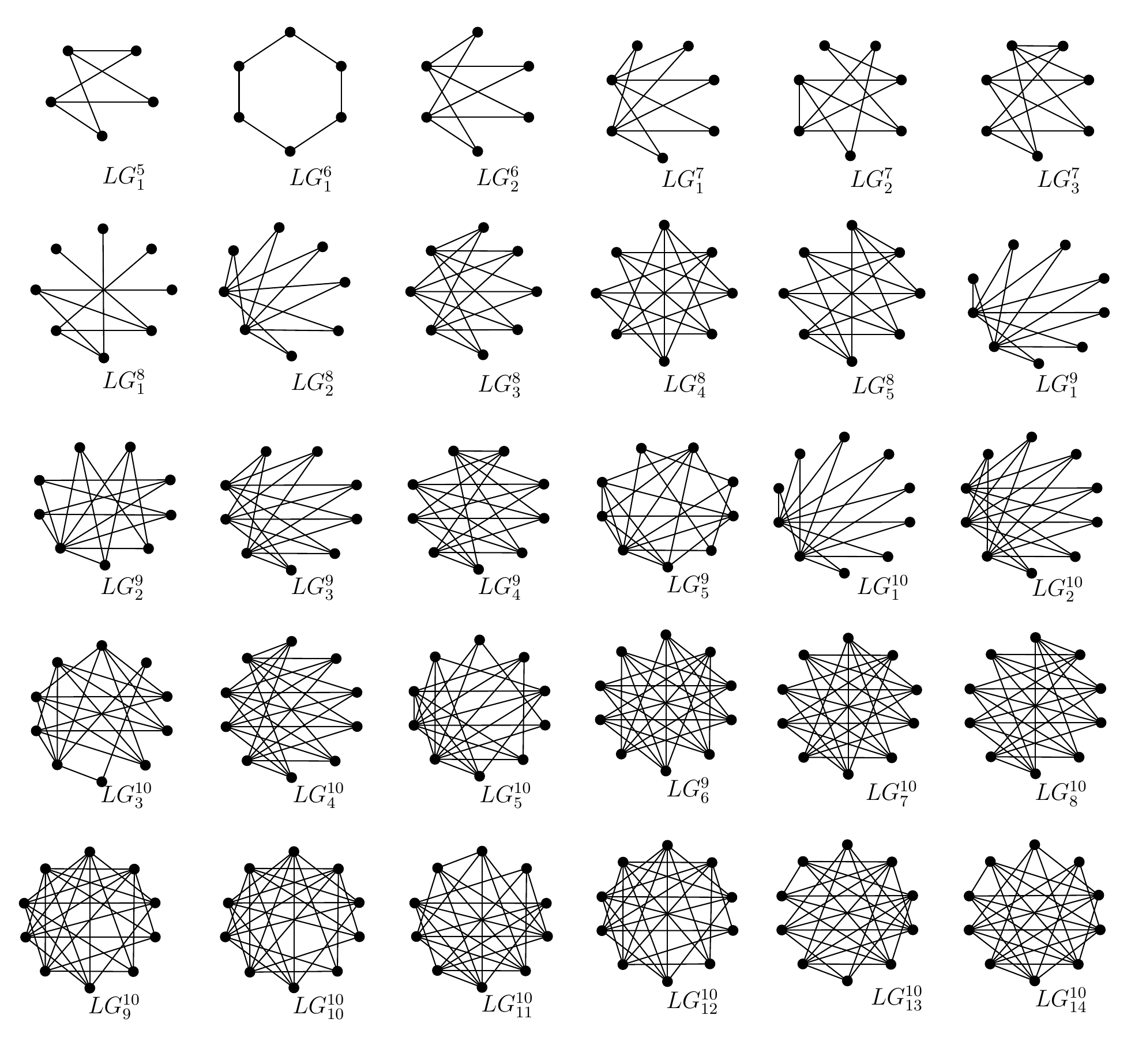}
   \caption{All the possible members of the graph class $\Omega$.}
    \label{fig6}
      \end{figure}

\scriptsize
\begin{longtable}[c]{|l|l|l|}
 \caption{Energies or the characteristic polynomials of the line graphs and complements of the line graphs, of some graphs of Figure \ref{fig6}.\label{Class6}}
\endfirsthead
\multicolumn{3}{c}{Continuation of Table \ref{Class6}}\\ \hline
\rule{0pt}{10pt} Graph $LG^{k}_i$&  Eigenvalues or Characteristic Polynomials of $L(LG^{k}_i)$ and $\displaystyle\overline{L(LG^{k}_i)}$ & Energies of $L(LG^{k}_i)$ and $\overline{L(LG^{k}_i)}$ \\
\endhead \hline
\endfoot
\hline
\rule{0pt}{10pt} Graph $LG^{k}_i$ &   Eigenvalues or Characteristic Polynomials of $L(LG^{k}_i)$ and $\overline{L(LG^{k}_i)}$ & Energies of $L(LG^{k}_i)$ and $\overline{L(LG^{k}_i)}$ \\ \hline

$LG^{6}_{1}$ & $-2, -1, -1, 1, 1, 2 $ & 8 \\

 & $-2, -2, 0, 0, 1, 3$ & \\ \hline

$LG^{7}_{2}$ & $1, \pm \sqrt 2 , \pm \sqrt 2 , - 2, - 2, - 2,\frac{{5 \pm \sqrt {17} }}{2}$ & $12+4\sqrt{2}$ \\

 & $5,1,1,1, - 2, - 2, - 1 \pm \sqrt 2 , - 1 \pm \sqrt 2$ & \\ \hline

$LG^{8}_{1}$ & $0, - 2,1 \pm \sqrt 5 , \pm \sqrt 2 , \pm \sqrt 2 $ & $2+2\sqrt{5}+4\sqrt{2}$ \\

 & $1, - 1,2 \pm \sqrt 5 , - 1 \pm \sqrt 2 , - 1 \pm \sqrt 2 $ & \\ \hline

$LG^{8}_{4}$ & $0,6,2,2,2,2, - 2, - 2, - 2, - 2, - 2, - 2, - 2, - 2,1 \pm \sqrt 5  $ & $30 + 2\sqrt{5}$\\
 & $9,1,1,1,1,1,1,1,1, - 1, - 3, - 3, - 3, - 3, - 2 \pm \sqrt 5 $ & \\ \hline

$LG^{9}_{2}$ & $3,1,1,1,1, - 1, - 1, - 2, - 2, - 2, - 2, - 2, - 2, - 2,\frac{{9 \pm \sqrt {41} }}{2}$ & 32 \\

            & $9,1,1,1,1,1,1,1,0,0, - 2, - 2, - 2, - 2, - 4, - 4$ & \\ \hline
									
$LG^{9}_{5}$ & $9,5,2,2,2,2,2,0,0, \underbrace{-2,\cdots,-2}_{\text{12 times}}$ & 48 \\
             & $12,\underbrace{1,\cdots,1}_{\text{12 times}}, - 1, - 1, - 3, - 3, - 3, - 3, - 4, - 6$ &  \\ \hline 			
													
$LG^{10}_{3}$ & $4 \pm 2\sqrt 3, \underbrace{-2,\cdots,-2}_{\text{11 times}}, -1,-1,2,2,2,2,4,4$ & 48 \\    			
              & $13, - 2,0,0,\underbrace{1,\cdots,1}_{\text{11 times}}, -3,-3,-3,-3,-5,-5$ &  \\   \hline 			

$LG^{10}_{5}$ & $5 \pm \sqrt {21} ,6,0,\underbrace{-2,\cdots,-2}_{\text{15 times}},1,1,2,2,4,4$ & 60 \\    			
              & $15, - 7, - 1,\underbrace{1,\cdots,1}_{\text{15 times}},-2,-2,-2,-3,-3,-5,-5 $ &             \\   \hline 			

$LG^{10}_{6}$ & $\frac{{5 \pm \sqrt {17} }}{2},8, - 1,\underbrace{-2,\cdots,-2}_{\text{15 times}},3,3,3,3,3,3$  & 62  \\    			
              & $\frac{{ - 7 \pm \sqrt {17} }}{2},16,0,\underbrace{1,\cdots,1}_{\text{15 times}},-4,-4,-4,-4,-4,-4 $ &     \\  \hline

$LG^{10}_{7}$ &  $\frac{{3 \pm \sqrt {33} }}{2},8,1,\underbrace{-2,\cdots,-2}_{\text{15 times}},3,3,3,3,3,3$  & $57+\sqrt{33}$ \\    			
              &  $\frac{{ - 5 \pm \sqrt {33} }}{2},16, - 2,\underbrace{1,\cdots,1}_{\text{15 times}},-4,-4,-4,-4,-4,-4 $ &        \\  \hline

$LG^{10}_{9}$ & ${(\lambda  - 2)^2}({\lambda ^2} - 10\lambda  + 8){({\lambda ^3} - 10{\lambda ^2} + 29\lambda  - 22)^2}{(\lambda  + 2)^{17}}$ &   \\

              & $(\lambda  - 17)(\lambda  + 2){(\lambda  + 3)^2}{({\lambda ^3} + 13{\lambda ^2} + 52\lambda  + 62)^2}{(\lambda  - 1)^{17}} $  &    \\ \hline
							
$LG^{10}_{10}$ & $\left\{ {\begin{array}{*{20}{c}}
  {(\lambda  - 2){{({\lambda ^3} - 10{\lambda ^2} + 31\lambda  -29)}}({\lambda ^2} - 10\lambda  + 8) } \\
  {\times({\lambda ^4} - 12{\lambda ^3} + 47{\lambda ^2} - 67\lambda  + 30)}(\lambda  + 2)^{17}
\end{array}} \right\}$ &   \\
               & $\left\{ {\begin{array}{*{20}{c}}
  {(\lambda  - 17)(\lambda  + 3)(\lambda  + 2)({\lambda ^3} + 13{\lambda ^2} + 54\lambda  + 71)} \\
  {\times({\lambda ^4} + 16{\lambda ^3} + 89{\lambda ^2} + 201\lambda  + 157){{(\lambda  - 1)}^{17}}}
\end{array}} \right\}$  &        \\ \hline

$LG^{10}_{11}$ & $\underbrace{-2,\cdots,-2}_{\text{17 times}},4,3,3,2,5 \pm \sqrt {17} ,\frac{{5 \pm \sqrt {17} }}{2},\frac{{7 \pm \sqrt {17} }}{2}$ & 68 \\  							 
		           & $17, - 5, - 3, - 2, - 4, - 4,\frac{{ - 7 \pm \sqrt {17} }}{2},\frac{{ - 9 \pm \sqrt {17} }}{2},\underbrace{1,\cdots,1}_{\text{17 times}}$ & \\ \hline
$LG^{10}_{12}$ & $\left\{ {\begin{array}{*{20}{c}}
  {(\lambda  - 2)({\lambda ^2} - 10\lambda  + 8)({\lambda ^3} - 8{\lambda ^2} + 17\lambda  - 8)} \\
  {\times(\lambda  - 4)({\lambda ^3} - 10{\lambda ^2} + 29\lambda  - 22){{(\lambda  + 2)}^{17}}}
\end{array}} \right\}$ &  \\

  		         & $\left\{ {\begin{array}{*{20}{c}}
  {(\lambda  + 5)(\lambda  + 3)(\lambda  + 2)({\lambda ^3} + 11{\lambda ^2} + 36\lambda  + 34)} \\
  {\times(\lambda  - 17)({\lambda ^3} + 13{\lambda ^2} + 52\lambda  + 62){{(\lambda  - 1)}^{17}}}
\end{array}} \right\}$ &     \\ \hline

$LG^{10}_{13}$ & $\underbrace{-2,\cdots,-2}_{\text{17 times}},5,4,4,3,2,1,5 \pm \sqrt {17} ,\frac{{5 \pm \sqrt {17} }}{2}$ & 68  \\  	
               & $17,-6,-5,-5,-4,-3,-2,-2,\frac{{ - 7 \pm \sqrt {17} }}{2},\underbrace{1,\cdots,1}_{\text{17 times}}$ &         \\  \hline	

$LG^{10}_{14}$ & $\underbrace{-2,\cdots,-2}_{\text{17 times}},4,3,3,2,5 \pm \sqrt {17} ,\frac{{5 \pm \sqrt {17} }}{2},\frac{{7 \pm \sqrt {17} }}{2}$ & 68 \\
							 & $17,-5,-4,-4,-3,-2,\frac{{ - 7 \pm \sqrt {17} }}{2},\frac{{ - 9 \pm \sqrt {17} }}{2},\underbrace{1,\cdots,1}_{\text{17 times}}$ &       \\  	\hline
\end{longtable}

\normalsize
\baselineskip=0.30in

We end this article by giving a result about a pair of complementary equienergetic graphs. For this, following Haemers \cite{Haemers}, we firstly state some definitions concerning designs.
A 2-($v, k, \lambda$) design, with parameters $v$, $k$ and $\lambda$, consists of a finite point set $\mathcal{P}$ of cardinality $v$ and a
collection $\mathcal{B}$ of subsets (called blocks) of $\mathcal{P}$, such that:\\
(i) each block has cardinality $k$ with the constraint $2 \le k \le v - 1$,\\
(ii) each (unordered) pair of points occurs in exactly $\lambda$ blocks.\\
A design in which number of blocks and number of points are same is called symmetric.
The incidence matrix $N$ of a 2-($v, k, \lambda$) design is the ($0, 1$) matrix with rows indexed by the points, and
columns indexed by the blocks, such that $N_{i,j} = 1$ if the point $i$ belongs to the block $j$, and $N_{i,j} =0$ otherwise. The incidence graph of a design with incidence matrix $N$ is the bipartite graph with the adjacency matrix
$
  \left(\begin{matrix}
                    O     &  N  \\[2mm]
                    N^T   &  O
                  \end{matrix}\right).
$
We also need the following two results.

\begin{theorem} {\rm{\cite{Cvetkovic-79}}}\label{Thm-AAA1}
If $IG(v, k, \lambda)$ is the incidence graph of the symmetric 2-($v, k, \lambda$) design then
$$
 Spec\left( {IG(v,k,\lambda )} \right) =
\left(\begin{matrix}
                   k   &  \sqrt {k - \lambda}    & - k      &  - \sqrt {k - \lambda}   \\[3mm]
                   1   &                v - 1    &   1      &      v - 1
                  \end{matrix}\right).
$$
\end{theorem}

\begin{theorem} {\rm \cite{Sachs-62}} \label{Thm-AAA0}
If $G$ is an $r$-regular graph of order $n$ with the eigenvalues $r,\lambda
_2, \cdots, \lambda_n$ then the eigenvalues of $\overline G$ are $n - r - 1,-\lambda_2 - 1, \cdots ,-\lambda_n - 1.$
\end{theorem}

We are now in position to state and prove the final result of this paper.

\begin{proposition}\label{prop-3}

For $l\ge4$, $IG(l,l - 1,l - 2)$ and $\overline {IG(l,l - 1,l - 2)}$ are complementary equienergetic graphs where $IG(l,l - 1,l - 2)$ is the incidence graph of the symmetric 2-$(l,l - 1,l - 2)$ design.

\end{proposition}

\begin{proof}
By Theorem \ref{Thm-AAA1}, it holds that
$$
 Spec\left( {IG(l,l-1,l-2)} \right) =
\left(\begin{matrix}
                   l-1   &     1      & 1-l     &     - 1   \\[3mm]
                     1   &  l- 1      &   1     &   l - 1
                  \end{matrix}\right).
$$
Also, by using Theorem \ref{Thm-AAA1}, we get
$$
 Spec\left(\overline {IG(l,l-1,l-2)} \right) =
\left(\begin{matrix}
                   l   &     -2   &  l - 2    &      0 \\[3mm]
                   1   &  l - 1   &      1    &  l - 1
                  \end{matrix}\right)
$$
and hence it holds that $\mathcal{E}\left( {IG(l,l - 1,l - 2)} \right) = \mathcal{E}\left( {\overline {IG(l,l - 1,l - 2)} } \right) = 4(l-1)$.

\end{proof}

\section*{Acknowledgement}
Suresh Elumalai's research is supported by University of Haifa, Israel, for the Postdoctoral studies and it is gratefully acknowledged.

\end{document}